\documentclass[10pt]{amsart}
\usepackage{fancyhdr}
\usepackage{stmaryrd}
\usepackage{calrsfs}

\usepackage{amsmath}
\usepackage[nospace,noadjust]{cite}
\usepackage{amsfonts}
\usepackage{tikz-cd}
\usepackage{amssymb,enumerate}
\usepackage{amsthm}
\usepackage{cite}
\usepackage{comment}
\usepackage{color}
\usepackage[all]{xy}
\usepackage{hyperref}
\usepackage{lineno}
\usepackage{stmaryrd}

\usepackage{mathtools}
\usepackage{bm}
\usepackage{graphicx}
\usepackage{psfrag}
\usepackage{url}

\usepackage{fancyhdr}

\usepackage{dirtytalk}
\usepackage{float}
\usepackage{mathrsfs}

\newtheorem*{maintheorem*}{Main Theorem}
\newtheorem{theorem}{Theorem}[section]
\newtheorem{prop}[theorem]{Proposition}
\newtheorem{question}[theorem]{Question}

\newtheorem{lemma}[theorem]{Lemma}
\newtheorem{cor}[theorem]{Corollary}

\theoremstyle{definition}
\newtheorem{definition}[theorem]{Definition}

\newtheorem{example}[theorem]{Example}
\numberwithin{equation}{section}

\newcommand{\nn}{\mathbb{N}}
\newcommand{\qq}{\mathbb{Q}}
\newcommand{\rr}{\mathbb{R}}
\newcommand{\zz}{\mathbb{Z}}

\newcommand{\pp}{\mathbb{P}}

\newcommand{\uu}{\mathcal{U}}
\newcommand{\mcd}{\text{mcd}}
\newcommand{\ord}{\text{ord}}
\newcommand{\red}{\text{red}}

\providecommand\ldb{\llbracket}
\providecommand\rdb{\rrbracket}

\hypersetup{
	pdftitle={FD},
	colorlinks=true,
	linkcolor=blue,
	citecolor=cyan,
	urlcolor=wine
}

\keywords{semidomain, polynomial semidomain, integral domain, Furstenberg semidomain, almost atomic semidomain, quasi-atomic semidomain, atomic, almost atomic, quasi-atomic}

\subjclass[2010]{Primary: 16Y60, 11C08; Secondary: 20M13, 13F05}

\begin{document}
	
	\mbox{}
	\title{On the subatomicity of polynomial semidomains}
	
	\author{Felix Gotti}
	\address{Department of Mathematics\\MIT\\Cambridge, MA 02139}
	\email{fgotti@mit.edu}
	
	\author{Harold Polo}
	\address{Department of Mathematics\\University of Florida\\Gainesville, FL 32611}
	\email{haroldpolo@ufl.edu}
	
\date{\today}

\begin{abstract}
	A semidomain is an additive submonoid of an integral domain that is closed under multiplication and contains the identity element. Although atomicity and divisibility in integral domains have been systematically investigated for more than thirty years, the same aspects in the more general context of semidomains have been considered just recently. Here we study subatomicity in the context of semidomains, focusing on whether certain subatomic properties ascend from a semidomain to its polynomial extension and its Laurent polynomial extension. We investigate factorization and divisibility notions generalizing that of atomicity. First, we consider the Furstenberg property, which is due to P. Clark and motivated by the work of H. Furstenberg on the infinitude of primes. Then we consider the almost atomic and quasi-atomic properties, both introduced by J. G. Boynton and J. Coykendall in their study of divisibility in integral domains.
\end{abstract}
\bigskip

\maketitle


\bigskip
\section{Introduction}
\label{sec:intro}

A semidomain is an additive submonoid of an integral domain that is closed under multiplication and contains a multiplicative identity element. Let $S$ be a semidomain, and set $S^* := S \setminus \{0\}$; that is, $S^*$ is the multiplicative monoid of~$S$. We say that $S$ is atomic provided that every non-invertible element of $S^*$ can be written as a finite product of atoms (i.e., irreducible elements). Factorizations in atomic domains have been systematically studied for more than three decades, considerably motivated by the landmark paper~\cite{AAZ90} by D. D. Anderson, D. F. Anderson, and M. Zafrullah. However, factorizations in the more general context of atomic semidomains have been investigated just recently by N. R. Baeth, S. T. Chapman, and the first author~\cite{BCG21}. In the present paper, we investigate atomic properties that are weaker than being atomic in the setting of semidomains. We put special emphasis on the ascent of such properties from the semidomain~$S$ to its polynomial extension $S[x]$ and its Laurent polynomial extension $S[x^{\pm 1}]$.
\smallskip

Special cases of polynomial semidomains and Laurent polynomial semidomains have been the focus of a great deal of attention lately in the factorization theory community. For instance, methods to factorize polynomials in the semidomain $\nn_0[x]$ were investigated by H. Brunotte in~\cite{hB13} and, more recently, F. Campanini and A. Facchini in~\cite{CF19} carried out a more systematic investigation of factorizations in $\nn_0[x]$. More generally, semigroup semirings were studied by V. Ponomarenko in~\cite{vP15} from the factorization perspective. The arithmetic of polynomial semidomains with coefficients in the semidomain $\rr_{\ge 0}$ has  also been considered: for instance, P. Cesarz, S. T. Chapman, S. McAdam, and G. J. Schaeffer in~\cite{CCMS09} studied the elasticity of $\rr_{\ge 0}[x]$.
\smallskip

Positive semirings, that is, subsemirings of $\rr_{\ge 0}$, have also been actively studied in the last few years. Factorizations in positive semirings consisting of rational numbers were considered in~\cite{CGG20a} by S.~T.~Chapman, M. Gotti, and the first author, and then in~\cite{ABP21} by S. Albizu-Campos, J. Bringas, and H.~Polo. The same semidomains were studied in~\cite{BG20} by Baeth and the first author in connection with factorizations of matrices. This, in turn, motivated the paper~\cite{BCG21} by Baeth, Chapman, and the first author, where several examples of positive semirings were constructed. Positive semirings can also be produced as valuations of polynomial semidomains and Laurent polynomial semidomains, and such valuations have also been investigated recently: the arithmetic of factorizations of $\nn_0[\alpha]$, where~$\alpha$ is a positive algebraic number, was studied recently for rational valuations in~\cite{CGG20a} by Chapman, Gotti, and the first author and for algebraic valuations in~\cite{CG22} by J. Correa-Morris and the first author, and in~\cite{JLZ22} by N. Jiang, B. Li, and S. Zhu. On the other hand, the atomic structure of the algebraic valuations of the Laurent polynomial semidomain $\nn_0[x^{\pm 1}]$ has been recently studied in~\cite{sZ22} by Zhu.
\smallskip

Following the terminology introduced by P. L. Clark in~\cite{pC17}, we say that the semidomain~$S$ is a Furstenberg semidomain if every nonunit element in $S^*$ is divisible by an atom. It follows from the definitions that each atomic semidomain is a Furstenberg semidomain. Furstenberg domains have been studied by N. Lebowitz-Lockard in~\cite{NL19} in connection with the properties of almost atomicity and quasi-atomicity, which we define in the next two paragraphs. In addition, Furstenberg domains have been recently considered in~\cite{GZ22} by the first author and Zafrullah in connection with idf-domains (i.e., integral domains whose elements have only finitely many irreducible divisors up to associates). Finally, Furstenberg domains have been considered in~\cite[Section~5]{GL22} by Li and the first author in the context of integer-valued polynomials. In Section~\ref{sec:Furstenberg}, we prove that the property of being Furstenberg ascends from the semidomain $S$ to both $S[x]$ and $S[x^{\pm 1}]$. We also construct an example of a Furstenberg semidomain that is neither an integral domain nor an atomic semidomain.
\smallskip

The semidomain $S$ is said to be almost atomic provided that, for every nonunit $b \in S^*$, there exist atoms $a_1, \dots, a_k$ of $S^*$ such that $a_1 \cdots a_k b$ factors into atoms in $S^*$. Observe that each atomic semidomain is almost atomic. The notion of almost atomicity was introduced in~\cite{BC15} by J. G. Boynton and J. Coykendall, and it was later studied by Lebowitz-Lockard~\cite{NL19} in parallel to various other subatomic properties. In Section~\ref{sec:almost atomicity}, we study almost atomicity in the context of semidomains. Unlike the case of the Furstenberg property, we do not know whether the property of being almost atomic ascends from the semidomain $S$ to either $S[x]$ or $S[x^{\pm 1}]$ (see Question~\ref{quest:ascending of almost atomicity}). However, under certain divisibility hypothesis on the coefficients of indecomposable polynomials, we are able to prove that almost atomicity ascends from the semidomain $S$ to both $S[x]$ and $S[x^{\pm}]$. We also provide in Section~\ref{sec:almost atomicity} an example of an almost atomic semidomain that is neither an integral domain nor an atomic semidomain. Finally, we exhibit an example of an antimatter semidomain that is not an integral domain whose polynomial extension is almost atomic.
\smallskip

As the notion of almost atomicity, that of quasi-atomicity was introduced in~\cite{BC15} and further studied in~\cite{NL19} in the context of integral domains. Following the terminology in~\cite{BC15}, we say that the semidomain~$S$ is quasi-atomic provided that, for every nonunit $b \in S^*$, there exists an element $c$ of~$S^*$ such that $bc$ factors into atoms in $S^*$. It follows directly from the definitions that each almost atomic semidomain is quasi-atomic. In Section~\ref{sec:quasi-atomicity}, we investigate quasi-atomicity in semidomains. We begin by constructing a quasi-atomic semidomain that is neither an integral domain nor an almost atomic semidomain. Then we provide a simple ideal-theoretical characterization of quasi-atomic semidomains. As for the property of being almost atomic, we could not determine whether the property of being quasi-atomic ascends from the semidomain $S$ to either $S[x]$ or $S[x^{\pm 1}]$. However, under the same hypothesis used to prove the ascent of almost atomicity, we prove that the property of being quasi-atomic ascends from the semidomain $S$ to both $S[x]$ and $S[x^{\pm 1}]$.

\bigskip
\section{Background}
\label{sec:background}

\medskip
\subsection{General Notation}

In this section, we introduce the notation and terminology necessary to follow our exposition. Reference material on factorization theory and semiring theory can be found in the monographs \cite{GH06} by A. Geroldinger and F. Halter-Koch and \cite{JG1999} by J. Golan, respectively. Throughout this paper, we let $\pp$, $\zz, \qq$, and $\rr$ denote the set of primes, integers, rational numbers, and real numbers, respectively. Additionally, we let $\nn_0$ and $\nn$ denote the set of nonnegative integers and positive integers, respectively. Given $s \in \rr$ and $X \subseteq \rr$, we set
\[
	X_{>s} \coloneqq \{r \in X \mid r>s\} \quad \text{ and } \quad X_{\ge s} \coloneqq \{r \in X \mid r \ge s\}. 
\]
For $m,n \in \zz$, we denote by $\llbracket m,n \rrbracket$ the closed discrete interval from $m$ to $n$; that is,
\[
	\llbracket m,n \rrbracket \coloneqq \{k \in \zz \mid m \leq k \leq n\}.
\]
Observe that $\ldb m, n \rdb$ is empty if $m>n$.

\medskip
\subsection{Monoids}

A \emph{monoid}\footnote{The standard definition of a monoid does not assume the cancellative and the commutative conditions.} is defined here to be a semigroup with an identity element that is both cancellative and commutative. Since our interest lies in the multiplicative structure of certain semirings, we will use multiplicative notation for monoids unless we specify otherwise. For the rest of this section, let $M$ be a monoid with identity~$1$. We set $M^{\bullet} \coloneqq M \setminus \{1\}$, and we let $\mathcal{U}(M)$ denote the group of units (i.e., invertible elements) of~$M$. In addition, we let $M_{\red}$ denote the quotient $M/\mathcal{U}(M)$, which is also a monoid. We say that $M$ is \emph{reduced} provided that $\mathcal{U}(M)$ is the trivial group, in which case, the monoids $M_{\red}$ and $M$ can be naturally identified (and we do so). The \emph{Grothendieck group} of~$M$, denoted here by $\mathcal{G}(M)$, is the abelian group (unique up to isomorphism) satisfying the property that any abelian group containing an isomorphic image of the monoid $M$ also contains an isomorphic image of the group $\mathcal{G}(M)$. The monoid $M$ is called \emph{torsion-free} provided that $\mathcal{G}(M)$ is a torsion-free group. For a subset $S$ of $M$, we let $\langle S \rangle$ denote the smallest submonoid of $M$ containing $S$, and if $M = \langle S \rangle$, then we say that~$S$ generates~$M$.

Puiseux monoids and positive monoids are used several time throughout this paper to construct needed examples. Following \cite{GG17}, we call any additive submonoid of $\qq_{\ge 0}$ a \emph{Puiseux monoid}. The class of Puiseux monoids is, therefore, a natural generalization of the class consisting of numerical monoids (i.e., additive submonoids of $\nn_0$ up to isomorphism). Puiseux monoids account up to isomorphism for all rank-$1$ torsion-free monoids that are not groups (see \cite[Theorem~3.12.1]{GGT21}). Following \cite{fG19}, we call any additive submonoid of $\rr_{\ge 0}$ a \emph{positive monoid}. It follows directly from the definitions that every Puiseux monoid is a positive monoid.  The atomic structure of both Puiseux and positive monoids has been actively investigated during the last few years: see the recent papers \cite{CGG21,GV23} as well as the references therein).

For $b,c \in M$, it is said that $c$ \emph{divides} $b$ \emph{in} $M$ if there exists $d \in M$ such that $b = cd$, in which case we write $c \mid_M b$, dropping the subscript precisely when $M = (\nn, \times)$. We say that $b,c \in M$ are \emph{associates} if $b \mid_M c$ and $c \mid_M b$. The monoid $M$ is called a \emph{valuation monoid} if for all $b,c \in M$ either $b \mid_M c$ or $c \mid_M b$. An element $p \in M \setminus \uu(M)$ is called \emph{prime} if for all $b,c \in M$ the relation $p \mid_M bc$ implies that either $p \mid_M b$ or $p \mid_M c$. A submonoid $N$ of $M$ is called \emph{divisor-closed} if for each $b \in N$ and $c \in M$ the relation $c \mid_M b$ implies that $c \in N$. Let $S$ be a nonempty subset of $M$. An element $d \in M$ is called a \emph{common divisor} of~$S$ provided that $d \mid_M s$ for all $s \in S$. A common divisor $d$ of $S$ is called a \emph{greatest common divisor} of~$S$ provided that $d$ is divisible by all common divisors of $S$. Also, a common divisor of~$S$ is called a \emph{maximal common divisor} provided that every common divisor of the set
\[
	S/d := \Big\{ \frac sd \ \Big{|} \ s \in S \Big\}
\]
belongs to $\uu(M)$; that is, if $d$ divides a common divisor $d'$ of $S$, then $d$ and $d'$ are associates in $M$. We let $\gcd_M(S)$ (resp., $\mcd_M(S)$) denote the set consisting of all greatest common divisors (resp., maximal common divisors) of $S$. The monoid $M$ is called a \emph{GCD-monoid} (resp., an \emph{MCD-monoid}) if each finite nonempty subset of $M$ has a greatest common divisor (resp., a maximal common divisor). It is clear that the inclusion $\gcd_M(S) \subseteq \mcd_M(S)$ holds, whence every GCD-monoid is an MCD-monoid. The converse does not hold in general, as the following example illustrates.

\begin{example}
	Let $M$ be the numerical monoid $\nn_0 \setminus \{1\}$; that is, $M = \langle 2,3 \rangle$. The set of common divisors of $\{5,6\}$ in $M$ is $\{0,2,3\}$. As $2 \nmid_M 3$ and $3 \nmid_M 2$, it follows that $\gcd_M(5,6)$ is empty. However, the only common divisor of both sets $\{5-2, 6-2\}$ and $\{5-3, 6-3\}$ is $0$ and, therefore, $\mcd_M(5,6) = \{2,3\}$. More generally, one can readily argue that every numerical monoid is an MCD-monoid.
\end{example}

An element $a \in M \setminus \uu(M)$ is called an \emph{atom} if for all $b,c \in M$ the equality $a = bc$ implies that either $b \in \uu(M)$ or $c \in \uu(M)$. One can readily verify that every prime element is an atom. We let $\mathcal{A}(M)$ denote the set consisting of all atoms of~$M$. Following P.~Cohn~\cite{pC68}, we say that $M$ is \emph{atomic} if every element of $M \setminus \uu(M)$ can be written as a (finite) product of atoms, while following J. Coykendall, D. Dobbs, and B. Mullins~\cite{CDM99}, we say that $M$ is \emph{antimatter} if $\mathcal{A}(M)$ is empty. One can readily check that~$M$ is atomic (resp., antimatter) if and only if $M_{\red}$ is atomic (resp., antimatter).

Assume throughout this paragraph that~$M$ is an atomic monoid. We let $\mathsf{Z}(M)$ denote the free (commutative) monoid on the set $\mathcal{A}(M_{\red})$. The elements of $\mathsf{Z}(M)$ are called \emph{factorizations}, and if $z = a_1 \cdots a_\ell \in \mathsf{Z}(M)$ for some $a_1, \ldots, a_\ell \in \mathcal{A}(M_{\red})$, then $\ell$ is called the \emph{length} of~$z$, which is denoted by $|z|$. Let $\pi \colon \mathsf{Z}(M) \to M_{\red}$ be the unique monoid homomorphism satisfying that $\pi(a) = a$ for all $a \in\mathcal{A}(M_{\red})$. For each $b \in M$, the sets
\begin{equation} \label{eq:sets of factorizations/lengths}
	\mathsf{Z}_M(b) \coloneqq \pi^{-1}(b \mathcal{U}(M)) \subseteq \mathsf{Z}(M) \hspace{0.6 cm}\text{ and } \hspace{0.6 cm}\mathsf{L}_M(b) \coloneqq \{|z| : z \in\mathsf{Z}_M(b)\} \subseteq \nn_0
\end{equation}
are of crucial importance to study the atomicity of $M$. When there seems to be no risk of ambiguity, we drop the subscript $M$ from the notations in~\eqref{eq:sets of factorizations/lengths}. We say that $M$ is a \emph{unique factorization monoid} (UFM) provided that $|\mathsf{Z}(b)| = 1$ for every $b \in M$. In addition, we say that $M$ is a \emph{bounded factorization monoid} (BFM) provided that $1 \le |\mathsf{L}(b)| < \infty$ for every $b \in M$. It follows from the corresponding definitions that every UFM is a BFM. 

Following the terminology in~\cite{pC17}, we say that the monoid $M$ is \emph{Furstenberg} provided that every nonunit element of $M$ is divisible by an atom. It follows directly from the definitions that every atomic monoid is Furstenberg. On the other hand, following the terminology in~\cite{BC15}, we say that the monoid $M$ is \emph{almost atomic} (resp., \emph{quasi-atomic}) provided that, for every nonunit $c \in M$, there exists $a_1, \dots, a_k \in \mathcal{A}(M)$ (resp., $b \in M$) such that $a_1 \cdots a_k c$ (resp., $bc$) can be written as a product of atoms in $M$. It follows directly from the definitions that every atomic monoid is almost atomic and also that every almost atomic monoid is quasi-atomic.

\medskip
\subsection{Semirings}

A \emph{semiring} $S$ is a nonempty set endowed with two binary operations denoted by `$+$' and `$\cdot$' and called \emph{addition} and \emph{multiplication}, respectively, such that the following conditions hold:
\begin{enumerate}
	\item $(S,+)$ is a monoid with its identity element denoted by $0$;
	\smallskip
	
	\item $(S, \cdot)$ is a commutative semigroup with an identity element denoted by $1$;
	\smallskip
	
	\item $b \cdot (c+d)= b \cdot c + b \cdot d$ for all $b, c, d \in S$.
\end{enumerate}
With notation as in the previous definition and for any $b,c \in S$, we write $b c$ instead of $b \cdot c$ when there seems to be no risk of confusion. It follows from conditions~(1) and~(3) in the definition of a semiring $S$ that $0 \cdot b = 0$ for all $b \in S$. A more general notion of a `semiring' $S$ does not assume that the semigroup $(S, \cdot)$ is commutative. However, this more general type of algebraic objects is not of interest in the scope of this paper. A subset $S'$ of a semiring $S$ is a \emph{subsemiring} of~$S$ if $(S',+)$ is a submonoid of $(S,+)$ that contains~$1$ and is closed under multiplication. Observe that every subsemiring of $S$ is a semiring.

\begin{definition}
	We say that a semiring $S$ is a \emph{semidomain} provided that $S$ is a subsemiring of an integral domain.
\end{definition}

Let $S$ be a semidomain. We set $S^* \coloneqq (S \setminus \{0\}, \cdot)$ and call it the \emph{multiplicative monoid} of $S$. It is worth emphasizing that a semiring $S$ may not be a semidomain even if~$S^*$ is a monoid; for instance, consider $\{(0,0)\} \cup (\nn \times \nn)$ under the usual component-wise addition and multiplication. Following standard notation from ring theory, we refer to the units of the multiplicative monoid $S^*$ simply as \emph{units} of $S$, and we denote the set of units of $S$ by $S^\times$. We never consider in this paper the units of the monoid $(S,+)$, so the use of the term `unit' in the context of the semidomain $S$ should not generate any ambiguity. We write $\mathcal{A}(S)$ instead of $\mathcal{A}(S^*)$ for the set of atoms of the multiplicative monoid $S^*$, while we let $\mathcal{A}_{+}(S)$ denote the set of atoms of the additive monoid $(S,+)$. In addition, for any $b,c \in S$ such that $b$ divides $c$ in $S^*$, we write $b \mid_S c$ instead of $b \mid_{S^*} c$. Finally, for any nonempty subset $T$ of $S^*$, we write $\gcd_S(T)$ (resp., $\mcd_S(T)$) instead of $\gcd_{S^*}(T)$ (resp., $\mcd_{S^*}(T)$). As for the notion of units, throughout this paper we never consider divisibility in the additive monoid of any semidomain.
\smallskip

\begin{lemma} \label{lem:characterization of integral semirings}
	For a semiring $S$, the following conditions are equivalent.
	\begin{enumerate}
		\item[(a)] The multiplication of $S$ extends to $\mathcal{G}(S)$ turning $\mathcal{G}(S)$ into an integral domain. 
		\smallskip
		
		\item[(b)] $S$ is a semidomain.
	\end{enumerate}
\end{lemma}

\begin{proof}
	(a) $\Rightarrow$ (b): This is clear.
	\smallskip
	
	(b) $\Rightarrow$ (a): Let $S$ be a semidomain, and suppose that $S$ is embedded into an integral domain~$R$. We can identify the Grothendieck group $\mathcal{G}(S)$ of $(S,+)$ with the subgroup $\{r - s \mid r,s \in S\}$ of the underlying additive group of $R$. It is routine to verify that $\mathcal{G}(S)$ is closed under the multiplication it inherits from $R$, and it contains the multiplicative identity because $0,1 \in S$. Hence $\mathcal{G}(S)$ is an integral domain having $S$ as a subsemiring. 
\end{proof}

We say that the semidomain $S$ is \emph{atomic} (resp., \emph{Furstenberg}, \emph{almost atomic}, \emph{quasi-atomic}, a \emph{GCD-semidomain}) if its multiplicative monoid $S^*$ is atomic (resp., Furstenberg, almost atomic, quasi-atomic, a GCD-monoid). Similarly, we say that~$S$ is a \emph{unique factorization semidomain} (UFS) provided that~$S^*$ is a UFM, and we say that $S$ is a \emph{bounded factorization semidomain} (BFS) provided that $S^*$ is a BFM. A subset $I$ of $S$ is an \emph{ideal}\footnote{Golan~\cite{JG1999} defines an ideal in a more restrictive way: if $I$ is an ideal of a semiring $S$, then by definition $I \neq S$. Consequently, any result we cite from \cite{JG1999} is interpreted here as a statement about the proper ideals of a semiring.} of $S$ provided that $(I,+)$ is a submonoid of $(S,+)$ and $I S \subseteq I$. We say that an ideal $I$ is \emph{prime} if $I \neq S$ and, for all $b,c \in S$, the containment $bc \in I$ implies that either $b \in I$ or $c \in I$. Although the semidomain $S$ can be embedded into an integral domain~$R$, the former may not inherit any (sub)atomic property from the latter as, after all, the integral domain~$\qq[x]$ is a UFD but it contains as a subring the integral domain $\zz + x\qq[x]$, which is not even quasi-atomic (see \cite[Lemma~17]{NL19}).

Let $S$ be a semiring. The set consisting of all polynomial expressions with coefficients in the semiring~$S$ is also a semiring, which we denote by $S[x]$ and call the \emph{semiring of polynomials over}~$S$. Additionally, if~$S$ is a semidomain embedded into an integral domain $R$, then it is clear that $S[x]$ is also a semidomain, and the elements of $S[x]$ are, in particular, polynomials in $R[x]$. Consequently, when $S$ is a semidomain all the standard terminology for polynomials can be applied to elements of $S[x]$, including \emph{constant polynomial}, \emph{degree}, \emph{order}, and \emph{leading coefficient}. Observe that $S^*$ is a divisor-closed submonoid of $S[x]^*$ and, therefore, $S[x]^\times = S^\times$ and $\mathcal{A}(S[x]) \cap S = \mathcal{A}(S)$. Following~\cite{mR93}, we say that a nonzero polynomial in $S[x]$ is \emph{indecomposable} if the polynomial cannot be written as a product of two nonconstant polynomials in $S[x]$.

Following the terminology introduced by Baeth, Chapman, and the first author in~\cite{BCG21}, we call a subsemiring of $\rr$ consisting of nonnegative numbers a \emph{positive semiring}. The fact that underlying additive monoids of positive semirings are reduced makes them more tractable. The reader can check~\cite{BCG21} for several examples of positive semirings. The class of semidomains clearly contains those of integral domains and positive semirings.

\bigskip
\section{Furstenbergness}
\label{sec:Furstenberg}

The Furstenberg property is a relaxation of the property of being atomic, and the reader can find interesting examples of non-atomic Furstenberg domains in \cite[Section~5]{GL22} and \cite[Section~4]{NL19}. We now construct an example of a Furstenberg positive semiring that is not atomic. In the construction, we use Lindemann-Weierstrass Theorem from transcendental number theory (see \cite[Chapter~1]{aB90}), which states that, for distinct algebraic numbers $\alpha_1, \ldots, \alpha_n$, the set $\{e^{\alpha_1}, \ldots, e^{\alpha_n}\}$ is linearly independent over the field of algebraic numbers.

\begin{example} \label{ex: non-atomic integral semiring that is Furstenberg}
	Consider the Puiseux monoid $P =  \big\langle \frac 1p \mid p \in \pp \big\rangle$, and set $M := P \cup \qq_{\ge 1}$. It is clear that~$M$ is also a Puiseux monoid. It is well known and not difficult to argue that $\mathcal{A}(P) = \big\{ \frac 1p \mid p \in \pp \big\}$ (see, for instance, \cite[Theorem~4.5]{AGH21}). This, along with the fact that no element of $\qq_{\ge 1}$ divides any of the elements of $\big\{ \frac1p \mid p \in \pp \big\}$ in $M$, implies that $\big\{ \frac 1p \mid p \in \pp \big\} \subseteq \mathcal{A}(M)$. In addition, observe that for any $q \in M_{\ge 1}$ we can pick $p \in \pp$ large enough so that $\frac 1p \mid_M q$. Putting the two previous observations together, we can conclude that
	\[
		\mathcal{A}(M) = \Big\{ \frac 1p \ \Big{|} \ p \in \pp \Big\}.
	\]
	This implies that $M$ is not atomic as, for instance, $\frac54$ cannot be written as a sum of atoms in $M$. On the other hand, it follows from our previous observations that $M$ is a Furstenberg monoid.
	
	Now consider the additive monoid $E(M) \coloneqq \langle e^m \mid m \in M \rangle$, which is free on the set $\{e^m \mid m \in M\}$ by Lindemann-Weierstrass Theorem. It is clear that $E(M) \subseteq \rr_{\ge 0}$. Also, observe that $E(M)$ contains $1$ and is closed under multiplication. As a consequence, $E(M)$ is a positive semiring. In addition, the fact that $M \subseteq \rr_{\ge 0}$ guarantees that $\min E(M)^* = 1$, which in turn implies that the multiplicative monoid $E(M)^*$ is reduced.
	
	We proceed to argue that $E(M)$ is a Furstenberg semidomain that is not atomic. It is clear that $e(M) \coloneqq \{e^m \mid m \in M\}$  is a multiplicative submonoid of $E(M)^*$ and also that $e(M)$ is isomorphic to the Puiseux monoid $M$. Since $M$ is not atomic, $e(M)$ is not atomic. As $M$ consists of algebraic numbers, it follows from \cite[Lemma~2.7]{BCG21}\footnote{Although \cite[Lemma~2.7]{BCG21} is stated for positive monoids, its proof requires the use of Lindemann-Weierstrass Theorem and, therefore, it requires that the positive monoid $M$ consists of algebraic numbers.} that $e(M)$ is a divisor-closed submonoid of $E(M)^*$. Therefore the semidomain $E(M)$ is not atomic. 
	
	To argue that $E(M)$ is Furstenberg, take a nonunit $r \in E(M)^*$ (that is, $r \in E(M)^* \setminus \{1\}$) and write $r = c_1 e^{q_1} + \dots + c_k e^{q_k}$, where $c_1, \ldots, c_k \in \nn$ and $q_1, \ldots, q_k \in M$.  We split the rest of the argument into the following two cases.
	\smallskip
	
	\textsc{Case 1:} There exists a positive common divisor $d$ of the elements $q_1, \ldots, q_k$ in $M$. In this case, we can factor $r$ in $E(M)^*$ as $r = e^d (c_1 e^{q_1 - d} + \dots + c_k e^{q_k - d})$. Because $d > 0$ and $M$ is a Furstenberg monoid, there exists $a \in \mathcal{A}(M)$ such that $a \mid_M d$. Therefore the fact that $e(M)$ is a divisor-closed submonoid of $E(M)^*$ guarantees that $e^a$ is an atom of $E(M)$ such that $e^a \mid_{E(M)} r$.
	\smallskip
	
	\textsc{Case 2:} The only common divisor of the elements $q_1, \ldots, q_k$ in $M$ is $0$. In this case, we can choose nonunits $s_1, \dots, s_m \in E(M)^*$ satisfying that $r = s_1 \cdots s_m$. For each $i \in \ldb 1,m \rdb$, let $\ell_i$ be the length of~$s_i$ in the underlying free commutative monoid of $E(M)$. Since no element of the form $e^d$ with $d \in M^{\bullet}$ divides any of the factors $s_1, \dots, s_m$ in $E(M)^*$, we see that $2 \le \ell_i \le s_i$ for every $i \in \ldb 1,m \rdb$. Hence from the fact that $2^m \le s_1 \cdots s_m = r$, we deduce that $m \le \log_2 r$. Now, we can assume that $m$ has been taken as large as it can possibly be to conclude that $s_1$ is an atom of $E(M)$ such that $s_1 \mid_{E(M)} r$.
	\smallskip
	
	In any case, $r$ is divisible by an atom in $E(M)$. Thus, we can conclude that $E(M)$ is a Furstenberg semidomain that is not atomic.
\end{example}

Next we prove that the Furstenberg property ascends from a semidomain to both its polynomial extension and its Laurent polynomial extension.

\begin{theorem} \label{theorem: Furstenberg}
	For a semidomain $S$, the following statements are equivalent.
	\begin{enumerate}
		\item[(a)] $S$ is Furstenberg.
		\smallskip
		
		\item[(b)] $S[x]$ is Furstenberg.
		\smallskip
		
		\item[(c)] $S[x^{\pm 1}]$ is Furstenberg.
	\end{enumerate}
\end{theorem}

\begin{proof}
	(a) $\Rightarrow$ (b): Suppose that $S$ is a Furstenberg semidomain. Take a nonzero nonunit $f \in S[x]$. Suppose first that $f \in S$. Then the fact that $S^*$ is a Furstenberg monoid guarantees the existence of $a \in \mathcal{A}(S)$ with $a \mid_{S} f$. As $S^*$ is a divisor-closed submonoid of $S[x]^*$, it follows that $a$ is also an atom of $S[x]$, and so $f$ is divisible by an atom in $S[x]$. Suppose now that $\deg f \ge 1$. Take the largest $m \in \nn$ such that $f = r g_1 \cdots g_m$ for some $r \in S^*$ and $g_1, \dots, g_m \in S[x]$ with $\deg g_i \ge 1$ for every $i \in \ldb 1,m \rdb$. If $g_1 \in \mathcal{A}(S[x])$ we are done. If $g_1$ is reducible, then the maximality of $m$ guarantees that $g_1 = s(g_1/s)$ for some nonunit element $s \in S^*$ dividing $g_1$ in $S[x]^*$. Because $s$ is a nonunit of $S^*$ and $S^*$ is a Furstenberg monoid, $s$ must be divisible by an atom $b$ in $S^*$. Since $S^*$ is a divisor-closed submonoid of $S[x]^*$, we see that $b$ is an atom of $S[x]$ that divides $f$ in $S[x]^*$. Hence $S[x]$ is also a Furstenberg semidomain.
	\smallskip
	
	(b) $\Rightarrow$ (c): First, observe that every atom $a$ in $S[x]$ with $\text{ord} \, a = 0$ (i.e., every atom not in the set $\{ux \mid u \in S^\times\}$) is an atom in $S[x^{\pm 1}]$. Now take a nonzero nonunit $g \in S[x^{\pm 1}]$, and write $g = x^d h$ for some $d \in \zz$ and $h \in S[x]$ with $\text{ord} \, h = 0$. As $g$ is not a unit in $S[x^{\pm 1}]$, we see that $h$ is not a unit in $S[x]$, and so there is an $a \in \mathcal{A}(S[x])$ such that $a \mid_{S[x]} h$. Note that  $\text{ord} \, a = 0$ because the same holds for~$h$. Thus, $a$ is an atom in $S[x^{\pm 1}]$ dividing $g$. Therefore $S[x^{\pm 1}]$ is also a Furstenberg semidomain.
	\smallskip
	
	(c) $\Rightarrow$ (a): This follows from the fact that $\{sx^n \mid s \in S^* \text{ and } n \in \zz\}$ is a divisor-closed submonoid of $S[x^{\pm 1}]^*$ whose reduced monoid is isomorphic to that of $S^*$.
\end{proof}

Observe that Theorem~\ref{theorem: Furstenberg} can help us identify Furstenberg semidomains that are not atomic. For instance, M. Roitman~\cite{mR93} provided the first example of an atomic domain $D$ such that $D[x]$ is not atomic. By virtue of Theorem~\ref{theorem: Furstenberg}, we can now assert that $D[x]$ is a non-atomic Furstenberg domain.

\bigskip
\section{Almost Atomicity}
\label{sec:almost atomicity}

In this section, we focus on the property of being almost atomic which, as that of being Furstenberg, is a property weaker than being atomic. As the next example illustrates, there are almost atomic semidomains (indeed positive semirings) that are not atomic. The notion of a greatest-divisor submonoid will play a useful role in the next example. Let $M$ be a monoid and let $N$ be a submonoid of $M$. For each $m \in M$, a \emph{greatest divisor} of $m$ in $N$ is an element $d \in N$ such that $d \mid_M m$ and if $d' \mid_M m$ for some $d' \in N$, then $d' \mid_N d$. Following~\cite{GL23}, we say that $N$ is a \emph{greatest-divisor submonoid} of $M$ provided that every element of $M$ has a greatest divisor in $N$.

\begin{example} \label{ex: almost atomic does not imply atomic}
	 Let $(p_n)_{n \ge 1}$ be the strictly increasing sequence consisting of all primes greater than $4$, and then consider the following sets: 
	\[
		A := \bigg\{ \frac{1}{p_n},\, \frac{1}{2^{n + 2}} + \frac{1}{2} - \frac{1}{p_n} \ \Big| \ n \in \nn \bigg\} \quad \text{ and } \quad B := \bigg\{ \frac1{2^{n + 2}} \ \Big| \ n \in \nn \bigg\}.
	\]
	Now set $M := \langle A \cup B \rangle$, and let us argue the following claim.
	\smallskip
	
	\noindent \textit{Claim 1:}  $\mathcal{A}(M) = A$.
	\smallskip
	
	\noindent \textit{Proof of Claim 1:} It is clear that none of the elements of $B$ is an atom of $M$. Therefore $\mathcal{A}(M) \subseteq A$. Before arguing the reverse inclusion, set
	\[
		a_n := \frac1{p_n} \quad \text{ and } \quad a'_n := \frac{1}{2^{n + 2}} + \frac{1}{2} - \frac{1}{p_n}
	\]
	for every $n \in \nn$. Fix $m \in \nn$. First, write $a_m = \sum_{i=1}^k (c_i a_i + c'_i a_i)$ for some $k \in \nn_{\ge m}$ and coefficients $c_1, \dots, c_k, c'_1, \dots, c'_k \in \nn_0$. Since $p_m > 4$, it follows that $a'_m > a_m$ and, therefore, $c'_m = 0$. Now, after applying $p_m$-adic valuation to both sides of the equality $a_m = \sum_{i=1}^k (c_i a_i + c'_i a_i)$, we find that $c_m \ge 1$, which implies that $c_m = 1$ and $c_i = c'_i = 0$ for every $i \in \ldb 1,k \rdb \setminus \{m\}$. Hence $a_m \in \mathcal{A}(M)$. Now write $a'_m = \sum_{i=1}^\ell (d_i a_i + d'_i a_i)$ for some $\ell \in \nn_{\ge m}$ and coefficients $d_1, \dots, d_\ell, d'_1, \dots, d'_\ell \in \nn_0$. If $d'_m = 0$, then we can write
	\begin{equation} \label{eq:auxy}
		\frac{1}{2^{m + 2}} + \frac{1}{2} - \frac{1 + d_m}{p_m} = \sum_{i \in \ldb 1, \ell \rdb \setminus \{m\}} \! \! \! (d_i a_i + d'_i a'_i).
	\end{equation}
	After applying $p_m$-adic valuation to both sides of the equality~\eqref{eq:auxy}, we find that $p_m \mid 1 + d_m$, which is not possible as $ \sum_{i \in \ldb 1, \ell \rdb \setminus \{m\}} (d_i a_i + d'_i a'_i) \ge 0$. Therefore $d'_m \ge 1$, which implies that $d'_m = 1$ and the rest of the coefficients in $\sum_{i=1}^\ell (d_i a_i + d'_i a_i)$ equal $0$. Hence $a'_m \in \mathcal{A}(M)$. As a result, the inclusion $A \subseteq \mathcal{A}(M)$ also holds, and the claim follows.
	\smallskip
	
	We can argue now that $M$ is almost atomic but not atomic. To show that $M$ is not atomic, first observe that for every $n \in \nn$ the fact that $p_n > 4$ implies that $a'_n > \frac1{2^{n+2}} + \frac14 > \frac14$. As a result, $\frac18$ cannot be divisible in $M$ by $a'_n$ for any $n \in \nn$. Therefore, if $M$ were atomic, then $\frac18 \in \big\langle \frac1{p_n} \mid n \in \nn \big\rangle$, which is not possible because every element in $\big\langle \frac1{p_n} \mid n \in \nn \big\rangle$ has nonnegative $2$-adic valuation. Thus,~$M$ cannot be atomic. Let us show now that $M$ is almost atomic. To do so, fix a nonzero $q \in M$. Now write $q = c \frac1{2^{n+1}} + q'$ for some $c \in \nn_0$, $n \in \nn$, and $q' \in M$ such that $q'$ can be written as a sum of atoms in $M$. Thus,
	\[
		c+q = 2c \bigg(\frac1{2^{n+2}} + \frac12 \bigg) + q' = 2c(a_n + a'_n) + q'.
	\]
	Therefore $c+q$ can be written as a sum of atoms in $M$. It is clear, on the other hand, that if $c \neq 0$, then $c \in \nn \subseteq M$ can also be written as a sum of atoms in $M$: indeed, $c = (c p_n) a_n$. As a consequence,~$M$ is almost atomic.
	
	Now set $N := \langle B \rangle$, and observe that $N$ is a valuation monoid: indeed, $N = \nn_0\big[ \frac12\big]$. Fix $q \in M$. Since $M = \langle A \cup B \rangle$, we can write each $q$ as follows:
	\begin{equation} \label{eq:canonical decomposition}
		q = c_0 + \sum_{n \in \nn} (c_n a_n + c'_n a'_n),
	\end{equation}
	for some $c_0 \in N$ and sequences $(c_n)_{n \ge 1}$ and $(c'_n)_{n \ge 1}$ of nonnegative integers, where all but finitely many terms of each sequence equal $0$. Among all decompositions as the one in~\eqref{eq:canonical decomposition}, assume we have chosen one minimizing $\sum_{n \in \nn} (c_n + c'_n)$. We claim that $c_n, c'_n \in \ldb 0, p_n - 1 \rdb$ for every $n \in \nn$. To argue this, note that if $c_j \ge p_j$ for some $j \in \nn$, then we could replace in~\eqref{eq:canonical decomposition} the terms $c_0$ and $c_j a_j$ by the terms $c_0 + p_j a_j$ and $(c_j - p_j) a_j$, respectively, to obtain another decomposition of $q$ whose existence violates the minimality of $\sum_{n \in \nn} (c_n + c'_n)$ (note that $p_j a_j \in N$). A similar argument shows that there cannot be any $j \in \nn$ with $c'_j \ge p_j$. Thus, a decomposition of $q$ as in~\eqref{eq:canonical decomposition} with $c_n, c'_n \in \ldb 0, p_n-1 \rdb$ for every $n \in \nn$ exists. We call such a decomposition a \emph{canonical decomposition} of $q$ and we call $c_0$ the \emph{dyadic summand} of the same canonical decomposition.
	\smallskip
	
	\noindent \textit{Claim 2:} Each $q \in M$ has finitely many canonical decompositions.
	\smallskip
	
	\noindent \textit{Proof of Claim 2:} This is clear when $q = 0$. Now pick an arbitrary nonzero element $q \in M$, and assume that $q$ has a canonical decomposition as that shown in~\eqref{eq:canonical decomposition}. Take $\ell \in \nn$ such that $p_\ell$ is greater than any prime dividing the denominator of~$q$. Fix $j \in \nn$, and observe that the summand $c_j a_j + c'_j a'_j$ is nonzero if and only if the $p_j$-adic valuation of $q$ is negative if and only if $p_j$ divides the denominator of~$q$: this is because the $p_j$-adic valuation of $c_n a_n + c'_n a'_n$ is nonnegative for every $n \neq j$. Therefore we can write the canonical decomposition in~\eqref{eq:canonical decomposition} as $q = c_0 + \sum_{n=1}^\ell (c_n a_n + c'_n a'_n)$. Since $c_0 \in N$ is completely determined by the sum $\sum_{n=1}^\ell (c_n a_n + c'_n a'_n)$, the restrictions $c_n, c'_n \in \ldb 0, p_n - 1 \rdb$ for every $n \in \ldb 1, \ell \rdb$ now imply the existence of only finitely many canonical decompositions for $q$. As $q$ was arbitrarily chosen, the claim is established.
	\smallskip
	
	\noindent \textit{Claim 3:} $N$ is a greatest-divisor submonoid of $M$.
	\smallskip
	
	\noindent \textit{Proof of Claim 3:} Fix $q \in M$. Observe that if $d \in N$ divides $q$ in $M$, then after adding $d$ to any canonical decomposition of $q-d$, we obtain a canonical decomposition of $q$ whose dyadic summand is at least $d$, and so divisible by $d$ in $N$ in light of the fact that the Puiseux monoid $N$ is a valuation monoid. Thus, every element of $N$ dividing $q$ in $M$ also divides the dyadic summand of some canonical decomposition of $q$. As a result, the fact that $q$ has only finitely many dyadic summands (by Claim~2) now implies that the maximum of all such dyadic summands is the greatest divisor of $q$ in $N$ (once again we are using that $N$ is a valuation monoid). Hence $N$ is a greatest-divisor submonoid of $M$, and the claim is established.
	\smallskip
	
	It is worth emphasizing that the fact that $M$ is a reduced monoid guarantees that every element of~$M$ has a unique greatest divisor in $N$.
	\smallskip
	
	\noindent \textit{Claim 4:} $M$ is an MCD-monoid.
	\smallskip
	
	\noindent \textit{Proof of Claim 4:} First, let us argue that if $q \in M$ and $d \in N$ such that $d$ is the greatest divisor of $q$ in~$N$, then the element $q - d$ has finitely many factorizations in $M$. Fix $q \in M$ and let $d \in N$ be the greatest divisor of $q$ in $N$. The statement follows immediately if $q \in N$. So assume that $q \in M \setminus N$, which implies that $q-d > 0$. Fix $j \in \nn$ such that $p_j$ does not divide the denominator of $q-d$, and let $c_j a_j + c'_j a'_j$ be a divisor of $q-d$ in $M$ maximizing the sum $c_j + c'_j$. As $a_j + a'_j \in N^\bullet$, the maximality of~$d$ guarantees that $c_j c'_j = 0$. Thus, if $c_j \ge 1$, then as the $p_j$-adic valuation of $q-d$ is nonnegative, it follows that $p_j \mid c_j$ and so $c_j a_j = (c_j - p_j) a_j + p_j a_j \in (c_j - p_j) a_j + N^\bullet$, which is not possible given the maximality of $d$. In the same way, we can argue that the inequality $c'_j \ge 1$ is not possible. Hence $c_j = c'_j = 0$, and so the maximality of $c_j + c'_j$ ensures that neither $a_j$ nor $a'_j$ divide $q-d$ in $M$. Thus, $q-d$ is divisible by only finitely many atoms in $M$, which implies that $q-d$ has finitely many factorizations in $M$.
	
	We are in a position now to show that $M$ is an MCD-monoid. To do this, suppose that $q_1, \dots, q_m$ are pairwise distinct nonzero elements in $M$. Now let $d_1, \dots, d_m \in N$ be the greatest divisors of $q_1, \dots, q_m$ in $N$, respectively. After relabeling if necessary, we can assume that $d_1 = \min\{d_i \mid i \in \ldb 1,m \rdb\}$. Since $N$ is a valuation monoid, it is clear that $d_1$ is a common divisor of $q_1, \dots, q_m$ in $M$. In addition, as $d_1$ is the greatest divisor of $q_1$ in $N$, the element $q_1 - d_1$ is not divisible by any elements of $N$ in $M$. This, together with the fact that $q_1 - d_1$ has only finitely many factorizations in $M$, implies that $q_1 - d_1$ has only finitely many divisors in $M$. Thus, after setting
	\[
		d_1' := \max \big\{ d \in M \mid \, d \mid_M q_i - d_1 \ \text{ for every } \ i \in \ldb 1,m \rdb \big\},
	\]
	we find that $d_1 + d'_1$ is a maximal common divisor of $q_1, \dots, q_m$ in $M$. Therefore $M$ is a an MCD-monoid, and the claim is established.
	\smallskip
	
	Finally, we will prove that the positive semiring $E(M) := \big\langle e^m \mid m \in M \big\rangle$ (constructed as in Example~\ref{ex: non-atomic integral semiring that is Furstenberg}) is almost atomic but not atomic. Since $M$ consists of algebraic numbers, it follows from \cite[Lemma~2.7]{BCG21} that $e(M) := \big\{ e^m \mid m \in M \big\}$ is a divisor-closed submonoid of the multiplicative monoid $E(M)^*$. Since $e(M)$ is isomorphic to $M$, which is not an atomic monoid, we find that $E(M)$ is not an atomic semidomain.
	
	It only remains to prove that $E(M)$ is an almost atomic semidomain. To do so, let $r$ be a nonunit element in $E(M)^*$, and write $r = c_1 e^{q_1} + \cdots + c_k e^{q_k}$, where $c_1, \ldots, c_k \in \nn$ and $q_1, \ldots, q_k \in M$ with $q_1 < \dots < q_k$. We split the rest of our argument into the following two cases.
	\smallskip
	
	\textsc{Case 1:} $k=1$. From the linear independence of the set $\big\{ e^m \mid m \in M \big\}$ over the field of algebraic numbers (due to Lindemann-Weierstrass Theorem), we can readily deduce that every (standard) prime is an atom of $E(M)^*$. Hence $c_1$ can be written as a product of atoms in $E(M)^*$. In addition, we can use the fact that $M$ is almost atomic to pick an element $b \in \langle \mathcal{A}(M) \rangle$ so that the element $b+q_1$ can be written as a sum of atoms in $M$. Therefore $e^b$ factors into atoms in $E(M)^*$ and, moreover, we can write $e^b r = e^b(c_1 e^{q_1}) = c_1 e^{b+ q_1}$ as a product of atoms in $E(M)^*$.
	\smallskip
	
	\textsc{Case 2:} $k \ge 2$. Since every (standard) prime is an atom of $E(M)^*$, we can assume, without loss of generality, that $c_1, \ldots, c_k$ are relatively prime positive integers. As $M$ is an MCD-monoid by Claim~4, we can take $d \in \mcd_M(q_1, \ldots, q_k)$. Now write
	\[
		r = e^d (c_1 e^{q_1 - d} + \dots + c_k e^{q_k - d}) = e^d s_1 \cdots s_m
	\]
	for some $m \in \nn$ and $s_1, \dots, s_m \in E(M)^* \setminus e(M)$ (this is possible because $k \ge 2$). Since no element of the form $e^{d'}$ with $d' \in M^{\bullet}$ divides any of the factors $s_1, \dots, s_m$ in $E(M)$, the inequality $2 \le s_i$ holds for every $i \in \ldb 1, m \rdb$. Therefore $2^m \le s_1 \cdots s_m = e^{-d} r$, which implies that $m \le \log_2 e^{-d} r$. Now, we can assume that~$m$ has been taken as large as it can possibly be to obtain that $s_1, \dots, s_m \in \mathcal{A}(E(M))$. Since~$M$ is almost atomic, there exists $b \in \langle \mathcal{A}(M) \rangle$ such that $b + d$ can be written as a sum of atoms in~$M$. As a consequence, both $e^b$ and $e^b r = e^{b+d}s_1 \cdots s_m$ factor into atoms in $E(M)$. Hence we conclude that $E(M)$ is almost atomic.
\end{example}

It was proved by Roitman in \cite[Proposition~1.1]{mR93} that the property of being atomic ascends from an integral domain to its polynomial extension provided that the coefficients of any indecomposable polynomial have a maximal common divisor (Roitman also proved in the same paper that the property of being atomic does not ascend, in general, from integral domains to their corresponding polynomial extensions). This result by Roitman was recently generalized by the authors to the context of semidomains (see \cite[Theorem~3.1]{GP23}). Under the same hypothesis, we can prove that the property of being almost atomic ascends from a semidomain to both its polynomial extension and its Laurent polynomial extension.

\begin{theorem} \label{thm:ascent of almost atomicity}
	For a semidomain $S$, each of the following statements implies the next one.
	\begin{enumerate}
		\item[(a)] $S$ is almost atomic and, for any indecomposable polynomial of the form $\sum_{i=0}^n c_i x^i \in S[x]^*$, the set $\emph{\mcd}_S(c_0, \dots, c_n)$ is nonempty.
		\smallskip
		
		\item[(b)] $S[x]$ is almost atomic.
		\smallskip
		
		\item[(c)] $S[x^{\pm 1}]$ is almost atomic. 
	\end{enumerate}
	Moreover, conditions (b) and (c) are equivalent.
\end{theorem}

\begin{proof}
	(a) $\Rightarrow$ (b): Let $f$ be a nonunit element of $S[x]^*$ such that $\deg f = n$ for some $n \in \nn_0$. Assume first that $n = 0$, which means that $f \in S^*$. In light of the almost atomicity of $S$, we can take an element $b_0 \in S^*$ that factors into atoms in $S$ such that $b_0f$ also factors into atoms in $S$. Since $\langle\mathcal{A}(S)\rangle \subseteq \langle\mathcal{A}(S[x])\rangle$, both $b_0$ and $b_0f$ factor into atoms in $S[x]$. Assume now that $n \ge 1$. Write $f = f_1 \cdots f_m$, where $f_i \in S[x]$ and $\deg f_i \ge 1$ for each $i \in \llbracket 1,m \rrbracket$. Suppose, without loss of generality, that $m$ is as large as it can possibly be. Fix an arbitrary $j \in \llbracket 1,m \rrbracket$. By the maximality of $m$, the polynomial $f_j$ is indecomposable in $S[x]$. Now write $f_j = \sum_{i=0}^n c_i x^i$ with coefficients $c_0, \ldots, c_n \in S$. Take $d_j  \in \mcd(c_0, \ldots, c_n)$, which exists by the assumed hypothesis, and note that $d_j^{-1}f_j \in \mathcal{A}(S[x])$. Since $S$ is almost atomic, there exists $b \in S^*$, which factors into atoms in $S$, such that $b d_1 \cdots d_m$ also factors into atoms in $S$. As $\langle\mathcal{A}(S)\rangle \subseteq \langle\mathcal{A}(S[x])\rangle$, both $b$ and $b d_1 \cdots d_m$ factor into atoms in $S[x]$. Therefore $bf = (b d_1 \cdots d_m) \prod_{i=1}^m d_i^{-1} f_i$, and so $bf$ factors into atoms in $S[x]$. Hence $S[x]$ is almost atomic.
	\smallskip
	
	(b) $\Rightarrow$ (c): First, observe that $\mathcal{A}(S[x]) \setminus \{ux \mid u \in S^\times\} \subseteq \mathcal{A}(S[x^{\pm 1}])$. Now let $f$ be a nonunit element of $S[x^{\pm 1}]^*$, and write $f = x^k g$ for some $k \in \zz$ and $g \in S[x]$ such that $\text{ord} \, g = 0$. Since $S[x]$ is almost atomic, there exists $b \in S[x]^*$, which factors into atoms in $S[x]$, such that $b g$ also factors into atoms in $S[x]$. Because $x$ is a prime element in $S[x]$, we can assume, without loss of generality, that $\text{ord} \, b = 0$. Since both $b$ and $bg$ factor into atoms in $S[x]$ and $\ord \, b = \text{ord} \, bg = 0$, the fact that $\mathcal{A}(S[x]) \setminus \{ux \mid u \in S^\times\}$ is a subset of $\mathcal{A}(S[x^{\pm 1}])$ guarantees that both $b$ and $bg$ also factor into atoms in $S[x^{\pm 1}]$. As a consequence, $S[x^{\pm 1}]$ is almost atomic.
	\smallskip
	
	(c) $\Rightarrow$ (b): Let $f$ be a nonunit element in $S[x]^*$. First, assume that $\text{ord} \, f = 0$. Because $S[x^{\pm 1}]$ is almost atomic, there exists $b \in S[x^{\pm 1}]$, which factors into atoms in $S[x^{\pm 1}]$ such that $bf$ is also factors into atoms in $S[x^{\pm 1}]$. As $x^k \in S[x^{\pm 1}]^{\times}$ for every $k \in \zz$, we may assume $\ord\, b = 0$, which implies that $\ord \, bf = 0$. If $b \in S[x^{\pm 1}]^{\times}$, then $b \in S^{\times}$, which trivially implies that $b$ factors into atoms in $S[x]$. Otherwise, we can write $b = a_1 \cdots a_n$ for some $a_1, \dots, a_n \in \mathcal{A}(S[x^{\pm 1}])$. Because $\text{ord} \, b = 0$, we can assume, without loss of generality, that $\ord \,a_i = 0$ for every $i \in \llbracket 1,n \rrbracket$. Therefore $a_1, \dots, a_n \in \mathcal{A}(S[x])$, and so $b$ factors into atoms in $S[x]$. As $\text{ord} \, bf = 0$, we can repeat the same argument to see that $bf$ factors into atoms in $S[x]$. 
	
	Finally, assume that $\text{ord} \, f = k \ge 1$. In this case, by the argument given in the previous paragraph, there exists $b \in S[x]$, which factors into atoms in $S[x]$, such that $b x^{-k} f$ also factors into atoms in $S[x]$. Since $x \in \mathcal{A}(S[x])$, both $b$ and $bf = x^k(bx^{-k} f)$ factor into atoms in $S[x]$.
	
	Hence we conclude that $S[x]$ is almost atomic.
\end{proof}

As a consequence of Theorem~\ref{thm:ascent of almost atomicity}, we see that, in the class consisting of all GCD-semidomains, the property of being almost atomic ascends to both polynomial extensions and Laurent polynomial extensions. In general, we do not know whether the polynomial extension of an almost atomic semidomain is almost atomic, so we pose the following question.

\begin{question} \label{quest:ascending of almost atomicity}
	Is there an almost atomic semidomain $S$ such that $S[x]$ is not almost atomic?
\end{question}

We conclude this section by providing an example of an antimatter semidomain $S$ whose polynomial extension $S[x]$ is almost atomic.

\begin{example}
	Consider the positive semiring $S = \{0\} \cup \qq_{\geq 1}$, which is antimatter~(see \cite[Example~3.10]{BCG21}). We shall prove that $S[x]$ is almost atomic. Take an arbitrary nonunit element $f \in S[x]^*$, and observe that we can write $f = cg$, where $c \in \qq_{\geq 1}$ and $g =  \sum_{i=0}^n c_i x^i \in S[x]^*$ with $c_j = 1$ for some $j \in \llbracket 0,n \rrbracket$. Therefore proving that $S[x]$ is almost atomic amounts to arguing that every element of $\qq_{\geq 1}$ and every nonconstant polynomial $\sum_{i=0}^n c_i x^i \in S[x]^*$ with $c_j = 1$ for some $j \in \llbracket 0,n \rrbracket$ can be expressed as quotients of two finite products of atoms of $S[x]$. Let us start with the latter case. To do so, fix a nonconstant polynomial $g = \sum_{i=0}^n c_i x^i \in S[x]^*$ with $c_j = 1$ for some $j \in \llbracket 0,n \rrbracket$. Then write $g = f_1 \cdots f_m$ as a product of indecomposable polynomials $f_1, \ldots, f_m \in S[x]$. Since $c_j = 1$ for some $j \in \ldb 0, n \rdb$, the only element of $S^* = \qq_{\ge 1}$ that divides all coefficients of $g$ must be $1$, and so the same follows for each of the factors $f_1, \dots, f_m$. This, along with the fact that the polynomials $f_1, \dots, f_m$ are indecomposable, ensures that $f_1, \dots, f_m \in \mathcal{A}(S[x])$. To tackle the former case, fix $c \in \qq_{\geq 1}$, and then write
	\begin{equation} \label{eq: decomposition}
		c = \frac{(cx + 1)(x + c)}{x^2 +\left(c + \frac{1}{c}\right)x + 1},
	\end{equation}
	where each of the polynomials in~\eqref{eq: decomposition} factors into atoms in $S[x]$ by the previous argument. Hence we conclude that $S[x]$ is almost atomic.
\end{example}

\bigskip
\section{Quasi-atomicity}
\label{sec:quasi-atomicity}

In this section, we provide an ideal-theoretical characterization of quasi-atomic semidomains, and then we study when quasi-atomicity ascends from a semidomain to its polynomial extension and its Laurent polynomial extension. Finally, we use quasi-atomicity to offer a stronger version of the known fact that every atomic GCD-domain is a UFD.
\smallskip

The fact that almost atomic semidomains are quasi-atomic follows immediately from the corresponding definitions. We proceed to construct a quasi-atomic semidomain that is neither an integral domain nor an almost atomic semidomain (the provided construction is based on that given in \cite[Example~7]{NL19} for integral domains). In our construction, we need the notion of a semifield. A \emph{semifield} is a semidomain in which every nonzero element has a multiplicative inverse. For instance, the semidomains $\qq_{\ge 0}$ and $\rr_{\ge 0}$ are both semifields.

\begin{example} \label{ex: quasi-atomic does not imply almost atomic}
 	Let $S$ be a BFS that is not a semifield (for instance,~$\nn_0$). Let $K$ be a field properly containing the field of fractions of $\mathcal{G}(S)$, and then consider the semidomain
 	\[
 		R := S[x] + x^2K[x] = S + Sx + x^2K[x].
 	\]
 	Observe that $R$ is an integral domain if and only if $S$ is an integral domain. Take an arbitrary polynomial $f = \sum_{i=0}^n c_i x^i \in R^*$, and set $m := \ord\, f$. We shall prove that $f$ factors into atoms in $R$ if and only if $c_m \in S$. 
	
	For the direct implication, assume that $c_m \not\in S$, and write $f = g_1 \cdots g_\ell$ with $g_1, \ldots, g_\ell \in R^*$. As $c_m \notin S$, we see that for some $j \in \llbracket 1,\ell \rrbracket$ the coefficient corresponding to the term $x^{\ord\, g_j}$ in $g_j$ is not an element of $S$. This implies that $\ord \, g_j \ge 2$. Thus, $g_j \notin S$ and every element of $S^*$ divides~$g_j$ in~$R$. Observe that $R^{\times} = S^{\times}$. Since $S$ is not a semifield, some nonunit of $S$ divides~$g_j$ in~$R$, and so $g_j \notin \mathcal{A}(R)$. Hence we conclude that $f$ cannot factor into atoms in $R$. 
	
	To argue the reverse implication, assume that $c_m \in S$, and then write $f = g_1 \cdots g_\ell$, where $g_i \not\in R^{\times}$ for any $i \in \llbracket 1,\ell \rrbracket$. Since $S$ is a BFS, we see that if $m = 0$ (resp., $m = 1$), then the inequality $\ell \le n + \max \mathsf{L}(c_0)$ (resp., $\ell \le n + \max \mathsf{L}(c_1)$) holds: indeed, for each $i \in \llbracket 1, \ell \rrbracket$, either $\deg g_i \ge 1$ or $g_i$ is a divisor of $c_0$ (resp., $c_1$) in $S$ that is not a unit. Consequently, if $m \in \{0,1\}$, then $f$ factors into atoms in~$R$. On the other hand, suppose that $m \ge 2$. Then write $f = x^{m - 1} c_m g$, where $g := x + \frac{c_{m+1}}{c_m}x^2 + \dots + \frac{c_n}{c_m} x^{n-m+1}$. Since $\ord \, g = 1$, we can mimic our previous argument to conclude that $g$ factors into atoms in $R$. In addition, because $S^*$ is a divisor-closed submonoid of $R^*$ and $S$ is a BFS, $c_m$ factors into atoms in $R$. These two last observations, together with the fact that $x \in \mathcal{A}(R)$, allow us to conclude that $f$ factors into atoms in $R$.
	
	Observe now that if $c_m \not\in S$, then $\frac{x^2}{c_m} f$ factors into atoms in $R$. This, along with the fact that $f$ factors into atoms when $c_m \in S$, implies that $R$ is quasi-atomic. On the other hand, if $c_m$ is not in the field of fractions of $\mathcal{G}(S)$, then for any $a_1, \dots, a_k \in \mathcal{A}(R)$ the element $h := a_1 \cdots a_k f$ does not factor into atoms in $R$ as the constant coefficient of $x^{-\text{ord} \, h} h$ does not belong to $S$. Consequently, the semidomain~$R$ is not almost atomic.
\end{example}

We turn to provide a characterization of quasi-atomic semidomains. To do so, we mimic the proof of \cite[Theorem~8]{NL19}.

\begin{theorem}
	A semidomain $S$ is quasi-atomic if and only if every nonzero prime ideal of $S$ contains an atom.
\end{theorem}

\begin{proof}
	For the direct implication, suppose that $S$ is quasi-atomic. Let $P$ be a nonzero prime ideal of $S$. Take a nonzero $r \in P$ (clearly, $r \not\in S^{\times}$). Since $S$ is quasi-atomic, there exist $b \in S^*$ and $a_1, \ldots, a_n \in \mathcal{A}(S)$ such that $a_1 \cdots a_n = b r \in P$. Because $P$ is a prime ideal, $a_i \in P$ for some $i \in \llbracket 1,n \rrbracket$. Thus, each nonzero prime ideal of $S$ contains an atom. 
	
	For the reverse implication, suppose that every nonzero prime ideal of $S$ contains an atom. Now assume towards a contradiction that $S$ is not quasi-atomic. Let $A$ be the multiplicative subset of $S$ consisting of all the elements that can be factored into atoms (this set includes all the units of $S$). As~$S$ is not quasi-atomic, we can pick $r \in S^*$ such that none of the elements in $Sr$ factors into atoms in~$S$. Thus, $Sr$ is a nonzero ideal of $S$ disjoint from $A$. Among all the ideals of~$S$ disjoint from $A$, let $P$ be maximal and, therefore, a nonzero ideal. By virtue of \cite[Proposition~7.12]{JG1999}, the ideal~$P$ is prime. Since~$P$ is disjoint from $A$, it contains no atoms, which is a contradiction.
\end{proof}

It turns out that under the same divisibility condition stated in Theorem~\ref{thm:ascent of almost atomicity} for the ascent of almost atomicity, the property of being quasi-atomic also ascends from a given semidomain to both its polynomial extension and its Laurent polynomial extension. This assertion is included in the following result.

\begin{theorem} \label{thm:ascent of quasi-atomicity}
	For a semidomain $S$, each of the following statements implies the next one.
	\begin{enumerate}
		\item[(a)] $S$ is quasi-atomic and, for any indecomposable polynomial of the form $\sum_{i=0}^n c_i x^i \in S[x]^*$, the set $\emph{\mcd}_S(c_0, \dots, c_n)$ is nonempty.
		\smallskip
		
		\item[(b)] $S[x]$ is quasi-atomic.
		\smallskip
		
		\item[(c)] $S[x^{\pm 1}]$ is quasi-atomic. 
	\end{enumerate}
	Moreover, conditions (b) and (c) are equivalent.
\end{theorem}

\begin{proof}
	(a) $\Rightarrow$ (b): Let $f$ be a nonunit element in $S[x]^*$, and set $n := \deg f$. We will argue the existence of $b \in S[x]^*$ such that $bf$ factors into atoms in $S[x]$. If $n = 0$, then our result follows immediately as~$S$ is quasi-atomic and $S^*$ is a divisor-closed submonoid of $S[x]^*$. Therefore we assume that $n \ge 1$. Write $f = f_1 \cdots f_m$, where $f_i \in S[x]$ and $\deg f_i \ge 1$ for every $i \in \llbracket 1,m \rrbracket$. Without loss of generality, assume that $m$ has been taken as large as it can possibly be. Fix an arbitrary $j \in \llbracket 1,m \rrbracket$. It follows from the maximality of $m$ that the polynomial $f_j$ is indecomposable. Now write $f_j = \sum_{i=0}^n c_i x^i $ for some coefficients $c_0, \ldots, c_n \in S$. By the assumed hypothesis, we can take $d_j \in \mcd_S(c_0, \ldots, c_n)$. Observe that $d_j^{-1}f_j \in \mathcal{A}(S[x])$. Since $S$ is quasi-atomic, there exists $b_j \in S^*$ such that $b_j d_j$ factors into atoms in~$S$. Now set $b := b_1 \cdots b_m$. It is clear that $b \in S[x]^*$ and, moreover, the equality $bf = \prod_{i=1}^m(b_i d_i) (d_i^{-1} f_i)$ illustrates that $bf$ factors into atoms in $S[x]$. Hence $S[x]$ is quasi-atomic.
	\smallskip
	
	(b) $\Rightarrow$ (c): Let $f$ be a nonunit element of $S[x^{\pm 1}]^*$, and write $f = x^k g$ for some $k \in \zz$ and $g \in S[x]$ such that $\text{ord} \, g = 0$. As $S[x]$ is quasi-atomic, there exists $b \in S[x]^*$ such that $b g$ can be written as a product of atoms in $S[x]$. Because $x$ is a prime element in $S[x]$, we can assume, without loss of generality, that $\text{ord} \, b = 0$, which implies that $\text{ord} \, bg = 0$. Therefore the fact that $\mathcal{A}(S[x]) \setminus \{ux \mid u \in S^\times\}$ is a subset of $\mathcal{A}(S[x^{\pm 1}])$ ensures that $bg$ factors into atoms in $S[x^{\pm 1}]$. Hence $S[x^{\pm 1}]$ is quasi-atomic.
	\smallskip
	
	(c) $\Rightarrow$ (b): Let $f$ be a nonunit element in $S[x]^*$. Assume first that $\text{ord} \, f = 0$. Since $S[x^{\pm 1}]$ is quasi-atomic, there exists $b \in S[x^{\pm 1}]^{*}$ such that $bf$ factors into atoms in $S[x^{\pm 1}]$. As $x^k \in S[x^{\pm 1}]^{\times}$ for every $k \in \zz$, we may assume $\ord\, b = 0$, which implies that $\ord\, bf = 0$. If $bf \in S[x^{\pm 1}]^{\times}$, then $bf \in S^{\times}$, and so $bf$ can be trivially written as a product of atoms in $S[x]$. Otherwise, we can write $bf = a_1 \cdots a_n$ for some $a_1, \dots, a_n \in \mathcal{A}(S[x^{\pm 1}])$. Since $\text{ord} \, bf = 0$, we can assume, without loss of generality, that $\ord \,a_i = 0$ for every $i \in \llbracket 1,n \rrbracket$. This implies that $a_i \in \mathcal{A}(S[x])$ for every $i \in \llbracket 1,n \rrbracket$. Hence $bf$ factors into atoms in $S[x]$. Finally, we can reduce the case where $\text{ord} \, f \ge 1$ to the case where $\text{ord} \, f = 0$ as we did in the proof of Theorem~\ref{thm:ascent of almost atomicity}. Hence $S[x]$ is quasi-atomic.
\end{proof}

As for the property of almost atomicity, we do not know in general whether the polynomial extension of a quasi-atomic semidomain is again quasi-atomic. 
\smallskip

Observe that, as an immediate consequence of Theorem~\ref{thm:ascent of quasi-atomicity}, inside the class of all GCD-semidomains quasi-atomicity ascends to both polynomial extensions and Laurent polynomial extensions. It is well known that a monoid is a UFM if and only if it is an atomic GCD-monoid. The following proposition gives another similar characterization of a UFM using quasi-atomicity.

\begin{prop} \label{prop: UFM characterization}
	Let $M$ be a monoid. Then $M$ is a UFM if and only if $M$ is a quasi-atomic GCD-monoid.
\end{prop}

\begin{proof}
	The direct implication follows immediately. As for the reverse implication, it is well known that an atomic GCD-monoid is a UFM (see, for example, \cite[Section~10.7]{fHK}). Thus, it suffices to show that~$M$ is atomic. To do so, let $b$ be a nonunit element of $M$. Since $M$ is quasi-atomic, there exists $c \in M^{\bullet}$ such that $bc$ factors into atoms in $M$. Write $bc = p_1 \cdots p_n$ for some $p_1, \ldots, p_n \in \mathcal{A}(M)$. It follows from \cite[Theorem~6.7(2)]{rG84} that $p_1, \dots, p_n$ are primes. Thus, for each $i \in \ldb 1,n \rdb$, either $p_i \mid_M b$ or $p_i \mid_M c$. Therefore $b$ must be the product of some of the factors $p_1, \dots, p_n$ up to associates. Hence $M$ is atomic.
\end{proof}

\begin{cor}
	Let $S$ be a semidomain. Then $S$ is a UFS if and only if it is a quasi-atomic GCD-semidomain.
\end{cor}

To ensure that a GCD-monoid is a UFM, some sort of subatomic property needs to be imposed. However, the property of being Furstenberg is not enough to guarantee that a GCD-semidomain is a UFS. This is illustrated in the following example.

\begin{example}\footnote{In this arXiv version, we have replaced the example discussed in the original paper, as published in Proceedings of the 2021 Graz Conference on Rings and Polynomials. The current example discusses the multiplicative structure of a semidomain, while the original example discusses the additive structure of a semidomain, which is less appropriate given the content of the paper. This is the only replacement in this arXiv version.}
	Let $M$ be the nonnegative cone of the totally ordered (additive) group $(\zz^2, \preceq)$, where~$\preceq$ denotes the lexicographical order with priority on the second coordinate:
	\[
		M := (\nn_0 \times \{0\}) \cup (\zz \times \nn).
	\]
	Since $M$ is the nonnegative cone of $\zz^2$, it is reduced. Moreover, for all $b,c \in M$ the divisibility relation $b \mid_M c$ holds if and only if $b \preceq c$ in $\zz^2$. Therefore the only atom of $M$ is the minimum of $M^\bullet$; that is, $\mathcal{A}(M) = \{(1,0)\}$. Set $a := (1,0)$. Because every element of $M^\bullet$ is divisible by $a$ in $M$, it follows that~$M$ is a Furstenberg monoid. On the other hand, $\nn_0 a  = \nn_0 \times \{0\} \subsetneq M$, and so $M$ is not atomic. Since $M$ is a valuation monoid, it is a GCD-monoid: indeed, $\gcd(b,c) = \min\{b,c\}$ for all $b,c \in M$.
	
	Let $\qq[x;M]$ denote the monoid ring of $M$ over $\qq$, that is, the ring consisting of all the polynomial expressions with exponents in $M$ and coefficients in $\qq$ (with addition and multiplication defined as for standard polynomials). Since $M$ is a torsion-free monoid, $\qq[x;M]$ is an integral domain and so a semidomain. Moreover, as $M$ is a GCD-monoid, it follows from \cite[Theorem~6.4]{GP74} that $\qq[x;M]$ is a GCD-domain. 
	
	Let us prove that $\qq[x;M]$ is also a Furstenberg domain. To do so, define a function $\varphi \colon \qq[x;M]^* \to \nn_0$ as follows: if $f := \sum_{i=1}^n q_i x^{(b_i, c_i)} \in \qq[x;M]$, where $(b_1, c_1) \prec \dots \prec (b_n, c_n)$ (and so $c_1 \le \dots \le c_n$), then set $\varphi(f) := c_n$. Observe that the set
	\[
		R := \{f \in \qq[x;M]^* \mid \varphi(f) = 0\}
	\]
	is a divisor-closed submonoid of the multiplicative monoid $\qq[x;M]^*$ that is isomorphic to the multiplicative monoid of the polynomial ring $\qq[x]$ (via the monoid homomorphism given by the natural assignments $x^{(b,0)} \mapsto x^b$ for every $b \in \nn_0$). Since $\qq[x]$ is a UFD, it follows that $R$ is a UFM. Now fix a nonunit $f \in \qq[x;M]^*$, and set $c := \varphi(f)$. Suppose first that $f$ is divisible in $\qq[x;M]$ by a nonunit element~$g$ with $\varphi(g) = 0$. As $g$ is also a nonunit in $R$, it follows that $g$ is divisible in $R$ by some $h \in \mathcal{A}(R)$, and so the fact that $R$ is a divisor-closed submonoid of $\qq[x;M]^*$ implies that $h$ is also an atom in $\qq[x;M]$. Thus, $h$ is an atom of $\qq[x;M]$ dividing $f$ in $\qq[x;M]$. Now we can suppose that~$f$ is not divisible by any nonunit $g \in \qq[x;M]$ with $\varphi(g) = 0$. Write $f = f_1 \cdots f_k$ for some nonunit elements $f_1, \dots, f_k \in \qq[x;M]$.  Observe that $k \le \sum_{i=1}^k \varphi(f_i) = \varphi(f)$. As a consequence, after assuming that $k$ has been taken as large as it can possibly be, we obtain that $f_1, \dots, f_k$ are atoms in $\qq[x;M]$. Therefore~$f$ is divisible by the atom $f_1$ in $\qq[x;M]$. Putting all together, we can conclude that $\qq[x;M]$ is a Furstenberg domain.
	
	Finally, let us show that the monoid ring $\qq[x;M]$ is not a UFD. To do this, first observe that $N := \{qx^m \mid q \in \qq^\times \text{ and } m \in M\}$ is a divisor-closed submonoid of the multiplicative monoid $\qq[x;M]^*$. Since the reduced monoid of $N$ is isomorphic to $M$, which is not atomic, it follows that $N$ is not atomic. Now the fact that $N$ is a divisor-closed submonoid of $\qq[x;M]^*$ implies that $\qq[x;M]$ is not atomic. 
	
	Summarizing, we have constructed a monoid ring $\qq[x;M]$ that is a Furstenberg GCD-domain but not a UFD.
\end{example}

\bigskip
\section*{Acknowledgments}

The authors would like to thank an anonymous referee for several comments and suggestions that helped improve an early version of this paper. During the preparation of this paper, the first author was supported by the NSF award DMS-1903069 and DMS-2213323, while the second author was partially supported by the University of Florida Mathematics Department Fellowship.

\bigskip

\end{document}